\documentclass[12pt,reqno]{amsart}
\usepackage{amsmath}
\usepackage{fullpage}
\theoremstyle{plain}
\newtheorem{thm}{Theorem}[section]

\theoremstyle{definition}
\newtheorem{defn}[thm]{Definition}
\newtheorem{ex}[thm]{Example}

\numberwithin{equation}{section}
\newcommand{\R}{\mathbb{R}}

\newcommand{\F}{\mathcal{F}}

 \usepackage{color}

\begin{document}
\title[...]{On fourth order retarded equations with functional boundary conditions: a unified approach}
%\date{today}

% ----------------------------------------------------------------

\subjclass[2020]{Primary 47H10, secondary 34K10, 34B10, 34B18}%
\keywords{Fixed point index, cone, Birkhoff--Kellogg type result, retarded functional differential equation, functional boundary condition.}%

\author[A. Calamai]{Alessandro Calamai}
\address{Alessandro Calamai, 
Dipartimento di Ingegneria Civile, Edile e Architettura,
Universit\`{a} Politecnica delle Marche
Via Brecce Bianche
I-60131 Ancona, Italy}%
\email{calamai@dipmat.univpm.it}%

\author[G. Infante]{Gennaro Infante}
\address{Gennaro Infante, Dipartimento di Matematica e Informatica, Universit\`{a} della
Calabria, 87036 Arcavacata di Rende, Cosenza, Italy}%
\email{gennaro.infante@unical.it}%

\begin{abstract}
By means of a recent Birkhoff-Kellogg type theorem, we discuss the solvability of a fairly general class of parameter-dependent fourth order retarded differential equations subject to functional boundary conditions. We seek solutions within a translate cone of nonnegative functions. We provide an example to illustrate our theoretical results.
\end{abstract}

\maketitle

\centerline{\it
Dedicated to Professor Jerome A.\ Goldstein in honor of his eightieth birthday}

\section{Introduction}

In this paper we investigate the existence of positive solutions of the 
following class of fourth order parameter-dependent functional differential equations with functional boundary conditions (BCs):
\begin{equation} \label{4ord-intro}
u^{(4)}(t)+\lambda F(t,u_t)=0,\ t \in [0,1],
\end{equation}
with initial conditions
\begin{equation}\label{kmt-intro}
u(t)=\psi(t),\ t \in [-r,0],
\end{equation}
and one of the following BCs
\begin{equation}\label{fbc-intro}
u^{(j)}(1)=\lambda B[u],
\end{equation}
where $j$ can be either $0$ or $1$, $2$, $3$. 

Fourth order ODEs with nonlocal and/or nonlinear boundary terms have been studied in the past, also in view of applications to mechanical systems,
 see for example the recent papers~\cite{CabJeb, genupa, giaml22, mayinzhang19, zhang22} and the references therein. 
Functional differential equations, in the case of fourth and higher order, have also been considered in the past, we mention here the manuscripts~\cite{BCFP17, jank}. 

We show that, under suitable assumptions, the functional boundary value problem (FBVP) \eqref{4ord-intro}--\eqref{kmt-intro}--\eqref{fbc-intro} admits a solution of the form
$(\bar\lambda,\bar u)$ with $\bar\lambda$ positive and $\bar u$ nontrivial and nonnegative. Under additional hypotheses on the nonlinearities we can prove that the FBVP~\eqref{4ord-intro}--\eqref{kmt-intro}--\eqref{fbc-intro} admits uncountably many solution pairs, this is illustrated in Example~\ref{examp}.
In some sense we extend the existence result in \cite{giaml22}, in which the author considered the ODE case, namely
$$
u^{(4)}(t) +\lambda f(t,u(t),u'(t),u''(t),u'''(t))=0, \quad t\in [0,1],
$$
under cantilever-type BCs.
We also mention the recent papers \cite{aklb21, li, mwylgl21}
for different and related existence results for fourth order ODEs with nonlocal boundary terms. We stress that the functional formulation $B$ in \eqref{fbc-intro}
is fairly general and can be used to deal with the interesting cases of nonlinear and nonlocal BCs; we refer a reader interested in these topics to the reviews~\cite{Cabada1, Conti, rma, Picone, sotiris, Stik, Whyburn} and the manuscripts~\cite{Chris-bj, kttmna, jw-gi-jlms}. Also note that, in our setting, the functional term $B$ is allowed to depend on the datum $\psi$, this is illustrated in~\eqref{dta-dep}.

Our approach is of topological nature and consists, roughly speaking, in rewriting the FBVP~\eqref{4ord-intro}--\eqref{kmt-intro}--\eqref{fbc-intro} into a perturbed Hammerstein 
integral equation of which we seek solutions in an appropriate subset of a Banach space.
As in the recent papers \cite{acgi16, acgi2} we work within the context of \emph{affine cones}. This setting seems to be quite natural for the case of delay equations subject to boundary conditions.
We apply a kind of Birkhoff-Kellogg type theorem in affine cones, obtained recently by the authors in \cite{acgi2}. This result can be considered as a complement of the interesting topological results in affine cones proved by Djebali and Mebarki in \cite{djeb2014}.

We stress that, in the context of higher-order functional differential equations, the choice of the functional space is not automatic, see for example the Introduction of \cite{BCFP17} for some remarks in this direction. If the fourth-order equation is associated to an initial-value problem, it seems to be convenient to work in the space $C^3([-r, 0], \R)$ like in \cite{BCFP17}. Here, on the other hand, when considering the BVP \eqref{4ord-intro}--\eqref{kmt-intro}--\eqref{fbc-intro}
 we fix $\psi \in C^2([-r, 0], \R)$ and look for solutions in the space $C^2([-r, 0], \R)$.

Our starting point to build the solution operator will be the data $\psi(0)$, $\psi'(0)$, $\psi''(0)$
 and the Green's functions associated to these data. This is a key difference with respect to the previous papers \cite{acgi16,acgi2}, where we assumed \textit{homogeneous} data for the function $\psi$ and its derivatives at $t_0=0$.
Another remarkable difference with \cite{acgi16,acgi2} is the following.
Here we work in an affine cone which is not ``centered'' in the initial datum $\psi$. Instead we consider a different translation obtained with a modified function $\widehat\psi$ which also takes into account the BC~\eqref{fbc-intro}: in fact, for each $j$ we construct an affine cone with a different vertex. This is also a crucial methodological difference with respect to the paper~\cite{giaml22}, where no translation is needed due to the different boundary data.

We close the paper with an illustrating specific example, in the spirit of \cite{acgi2, giaml22}, concerning 
the particular case of a parametrized fourth-order delay differential equation with three, possibly different, time-lags:
 \[
u^{(4)}(t)+\lambda f(t,u(t),u'(t),u''(t),u(t-r_0), u'(t-r_1),u''(t-r_2)),\ t \in [0,1].
\]

\section{Setting of the problem} 

Let $I\subset \R$ be a compact real interval. By $C^2(I, \R)$ we denote the Banach space of the twice
continuously differentiable functions defined on $I$ with the norm $$\|u\|_{I,2}:=\max\{\|u\|_{I,\infty},\|u'\|_{I,\infty},\|u''\|_{I,\infty}\},$$ where
$\|u\|_{I,\infty}:=\sup_{t\in I}|u(t)|$.

In the paper we use the following notation, which is standard in retarded functional differential equations (cfr.\ \cite{halelunel}).
Given a positive real number $r>0$, a continuous function $u: J \to \R$, defined on a real interval
$J$, and given any $t \in \R$ such that $[t-r, t] \subseteq J$, by
$u_t : [-r, 0] \to \R$ we mean the function defined by
$u_t (\theta) = u(t + \theta)$.

As pointed out in the Introduction, we study parametrized fourth order functional differential equations of type
\begin{equation} \label{eq-4ord}
u^{(4)}(t)+\lambda F(t,u_t)=0,\ t \in [0,1],
\end{equation}
with initial conditions
\begin{equation}\label{eq-ic}
u(t)=\psi(t),\ t \in [-r,0],
\end{equation}
together with one of the following functional (non necessarily local) BCs
\begin{equation}\label{eq-fbc-all}
u^{(j)}(1)=\lambda B[u],
\end{equation}
where can be either $0$ or $1$, $2$, $3$.

In the BVP~\eqref{eq-4ord}--\eqref{eq-ic}--\eqref{eq-fbc-all},
$\lambda$ is a nonnegative parameter, 
$B$ is a suitable positive continuous functional, to be defined later, while the initial datum
$$\psi: [-r,0] \to [0,+\infty)$$ is a given function  twice continuously differentiable and such that
 $\psi'(0)$, $\psi''(0)$ are nonnegative.

Regarding the operator $F: [0,1] \times C^2([-r, 0], \R) \to [0,\infty)$,
throughout the paper we will assume the following
Carath\'eodory-type conditions (see also \cite{halelunel}):
\begin{itemize}
\item  for each $\phi$, $t \mapsto F(t,\phi)$ is measurable;
\item  for a.\,e. $t$, $\phi \mapsto F(t,\phi)$ is continuous;
\item for each $R>0$, there exists $\varphi_{R} \in
L^{\infty}[0,1]$ such that{}
\begin{equation*}
F(t,\phi) \le \varphi_{R}(t) \ \text{for all} \ \phi \in C^2([-r, 0], \R)
\ \text{with} \ \|\phi\|_{[-r,0],2} \le R,\ \text{and a.\,e.}\ t\in [0,1].
\end{equation*}{}
\end{itemize}

 \subsection{A Birkhoff-Kellogg type theorem in affine cones}

In this section we recall a Birkhoff–Kellogg type result in affine cones recently proved in \cite{acgi2}. The proof relies on classical fixed point index theory for compact maps. We refer a reader interested in the fixed point index to the review of Amann \cite{amann} and to the book by  Guo and Lakshmikantham \cite{guolak}.

Let us firstly recall some useful notation.
Let $(X,\| \, \|)$ be a real Banach space. A \emph{cone} $K$ of $X$  is a closed set with $K+K\subset K$, $\mu K\subset K$ for all $\mu\ge 0$ and $K\cap(-K)=\{0\}$.
For $y\in X$, the \emph{translate} of the cone $K$ is defined as
$$
K_y:=y+K=\{y+x: x\in K\}.
$$
Given a bounded and open (in the relative
topology) subset $\Omega$  of $K_y$, we denote by $\overline{\Omega}$ and $\partial \Omega$
the closure and the boundary of $\Omega$ relative to $K_y$.
Given an open
bounded subset $D$ of $X$ we denote $D_{K_y}=D \cap K_y$, an open subset of
$K_y$.

We can now state the Birkhoff–Kellogg type result in affine cones.

\begin{thm}[\cite{acgi2}, Corollary 2.4] \label{BK-transl-norm}
Let $(X,\| \, \|)$ be a real Banach space, $K\subset X$ be a cone and
 $D\subset X$ be an open bounded set with $y \in D_{K_y}$ and
$\overline{D}_{K_y}\ne K_y$. Assume that $\F:\overline{D}_{K_y}\to K$ is
a compact map and assume that 
$$
\inf_{x\in \partial D_{K_y}}\|\F (x)\|>0.
$$
Then there exist $x^*\in \partial D_{K_y}$ and $\lambda^*\in (0,+\infty)$ such that $x^*= y+\lambda^* \F (x^*)$.
\end{thm}

\section{The associated Green's functions and existence results}

In order to illustrate our strategy, let us first focus on the following case, in which the nonlocal BC \eqref{fbc-intro} involves the third derivative at $t=1$; namely:
\begin{equation} \label{eq.BC3}
 \begin{cases}
u^{(4)}(t)+\lambda F(t,u_t)=0,\ & t \in [0,1], \\
u(t)=\psi(t),\  & t \in [-r,0], \\
u'''(1)=\lambda B[u],
 \end{cases}
\end{equation}
By means of a superposition principle, we associate to \eqref{eq.BC3} a \emph{perturbed} Hammerstein integral equation, in the spirit of \cite{acgi16, gi-wil, gi-jw-ems-06}.
Recall that $\psi \in C^2([-r, 0], \R)$, so the ``functional'' initial condition implies that
$u(0)=\psi(0)$, $u'(0)=\psi'(0)$, $u''(0)=\psi''(0)$.

First observe that, given a continuous function $y$, the BVP
\begin{equation*} \label{eq.BC3-omog}
 \begin{cases}
-u^{(4)}(t)=y(t),\  t \in [0,1], \\
u(0)=u'(0)=u''(0)=u'''(1)=0,
 \end{cases}
\end{equation*}
has the unique solution
$$
u(t)= \int_{0}^{1} k(t,s)y(s)ds,
$$
where the Green's function is, see for example \cite{zhang22},
$$
k(t,s)= \frac{1}{6}\begin{cases}t^3,\ &0\leq t\leq s\leq 1,\\ s(3t^2-3ts+s^2),\ &0\leq s\leq t\leq 1,
\end{cases}
$$

In particular, for the kernel $k$ the following positivity property holds:
$$
k(t,s) \geq 0\ \text{on}\ [0,1]\times [0,1],  
$$

A direct computation shows that the following problems have the corresponding solutions: The BVP
\begin{equation*} \label{eq.BC3-gamma0}
 \begin{cases}
u^{(4)}(t)=0,\  t \in [0,1], \\
u(0)=1, \ u'(0)=u''(0)=u'''(1)=0,
 \end{cases}
\end{equation*}
is solved by the constant
$\gamma_0(t)=1$; the BVP
\begin{equation*} \label{eq.BC3-gamma1}
 \begin{cases}
u^{(4)}(t)=0,\  t \in [0,1], \\
u'(0)=1, \ u(0)=u''(0)=u'''(1)=0,
 \end{cases}
\end{equation*}
is solved by 
$\gamma_1(t)=t$;  the BVP
\begin{equation*} \label{eq.BC3-gamma2}
 \begin{cases}
u^{(4)}(t)=0,\  t \in [0,1], \\
u''(0)=1, \ u(0)=u'(0)=u'''(1)=0,
 \end{cases}
\end{equation*}
is solved by
$\gamma_2(t)=\frac12 t^2$;
while for the condition at $t=1$ we have
\begin{equation*} \label{eq.BC3-gamma3}
 \begin{cases}
u^{(4)}(t)=0,\  t \in [0,1], \\
u(0)=u'(0)=u''(0)=0, \ u'''(1)=1,
 \end{cases}
\end{equation*}
which is solved by
$\gamma_3(t)=\frac16 t^3$.

Observe that $\gamma_i(t) \geq 0\ \text{on}\ [0,1]$, for $i=0,\dots,3$.
Moreover, one can check by direct computation that the maps  
$$t\mapsto k(t,s)H(t) \quad t\mapsto \gamma_3(t)H(t),  \quad t \in [-r,1] $$ 
are of class $C^2([-r, 1], \R)$,
where 
$$
H(\tau)=
\begin{cases}
1,\ & \tau \geq 0, \\0,\ & \tau<0.
\end{cases}
$$

Now, define
\begin{equation} \label{def-psihat-new}
\widehat\psi(t) :=
\begin{cases}
\psi(t),\ & t \leq 0, \\
\gamma_0(t) \psi(0)
+ \gamma_1(t) \psi'(0) + \gamma_2(t) \psi''(0),\ & t > 0.
\end{cases}
\end{equation}
and observe that, being $\psi$ of class $C^2$, the function $\widehat\psi$ has the same regularity.

Due to the above setting, the FBVP \eqref{eq.BC3}
can be rewritten in the following form:
\begin{equation}\label{eqhamm-0}
u(t)=\widehat\psi(t) + \lambda \Bigl( \int_{0}^{1} k(t,s)H(t)F(s,u_s)\,ds +  H(t)\gamma_3(t) B[u] \Bigr),\quad t \in [-r, 1]
\end{equation}

By $K_0$ we denote the following cone of non-negative functions in the Banach space $C^2([-r, 1], \R)$:
$$
K_0=\{u\in C^2([-r, 1], \R): u(t)\geq 0\ \forall t\in[-r,1],\ \text{and}\ u(t)=u'(t)=u''(t)= 0 \ \forall t\in[-r,0]\}.
$$
Observe that that the function 
$$w(t)=\begin{cases}
0,\ & t \in [-r,0], \\
t^3,\ & t \in [0,1],
\end{cases}
$$
belongs to $K_0$, hence $K_0 \neq \{0\}$.

For a given $\Psi \in  C^2([-r, 1], \R)$, we let $K_\Psi$ be the following translate of the cone $K_0$
$$K_\Psi=\Psi + K_0 = \{\Psi +u : u \in K_0\}.$$
\begin{defn}
Given $\Psi \in  C^2([-r, 1], \R)$ and $\rho>0$, we define the following subsets of $C^2([-r, 1], \R)$: 
$$K_{0,\rho}:=\{u\in K_0: \|u\|_{[0,1],2} <\rho\},\quad 
K_{\Psi,\rho}:= \Psi + K_{0,\rho}.$$
\end{defn}

As a consequence of the above Corollary \ref{BK-transl-norm}, we get the following existence result.

\begin{thm}\label{eigen}Let $\rho \in (0,+\infty)$ and assume the following conditions hold. 

\begin{itemize}

\item[$(a)$] 
There exist $\underline{\delta}_{\rho} \in C([0,1],\R_+)$ such that
\begin{equation*}
F(t,\phi)\ge \underline{\delta}_{\rho}(t),\ \text{for every}\ (t,\phi)\in [0,1] \times
C^2([-r, 0], \R)
\ \text{with} \ \|\phi\|_{[-r,0],2} \le \rho+\|\widehat\psi\|_{[-r,1],2}.
\end{equation*}

\item[$(b)$] 
$B: \overline K_{\widehat\psi,\rho} \to \R_{+}$ is continuous and bounded, in particular let $\underline{\eta}_{\rho}\in [0,+\infty)$ be such that 
\begin{equation*}
B[u]\geq \underline{\eta}_{\rho},\ \text{for every}\ u\in \partial K_{\widehat\psi,\rho}.
\end{equation*}
\item[(c)] 
The inequality 
\begin{equation}\label{condc}
\sup_{t\in [0,1]}\Bigl\{ \gamma_3(t)\underline{\eta}_{\rho}+\int_{0}^{1}  k(t,s) \underline{\delta}_{\rho} (s)\,ds\Bigr\}>0 
\end{equation}
holds.
\end{itemize}
Then there exist $\lambda_\rho\in (0,+\infty)$ and $u_{\rho}\in \partial K_{\widehat\psi,\rho}$ that satisfy the BVP~\eqref{eq.BC3}.
\end{thm}

\begin{proof}
Let
$$\F u(t):= \int_{0}^{1} k(t,s)H(t)F(s,u_s)\,ds +  H(t)\gamma_3(t) B[u]. $$
Observe that the operator $\F$ maps
$\overline{K}_{\widehat\psi,\rho}$ into $K_0$ and is compact. The compactness of  the Hammerstein integral operator is a consequence of the regularity assumptions on the terms occurring in it combined with a careful use of  the Arzel\`{a}-Ascoli theorem (see~\cite{Webb-Cpt}), while the perturbation term
$H(t)\gamma_3(t) B[ \cdot ]$ is a finite rank operator. 

Take $u\in \partial K_{\widehat\psi, \rho}$. Then we have

\begin{multline}\label{est1}
\|  \F u\|_{[-r, 1],2}\geq   \|  \F u\|_{[-r, 1],\infty}=\sup_{t\in [0, 1]} \Bigl| \int_{0}^{1} k(t,s)F(s,u_s)\,ds +  \gamma_3(t) B[u]\Bigr |\\
\geq \sup_{t\in [0,1]}\Bigl\{ \gamma_3(t)\underline{\eta}_{\rho}+\int_{0}^{1}  k(t,s) \underline{\delta}_{\rho} (s)\,ds\Bigr\}.
\end{multline}

Note that the RHS of \eqref{est1} does not depend on the particular $u$ chosen. Therefore we have
$$
\inf_{u\in \partial K_{\widehat\psi, \rho}}\|  \F u\|_{[-r, 1],2} \geq \sup_{t\in [0,1]}\Bigl\{ \gamma_3(t)\underline{\eta}_{\rho}+\int_{0}^{1}  k(t,s) \underline{\delta}_{\rho} (s)\,ds\Bigr\}>0,
$$
and the result follows by Theorem~\ref{BK-transl-norm}.
\end{proof}

\subsection{A unified approach}
Our purpose is now to show that, in the case of the BC~\eqref{eq-fbc-all} with $j=0,1,2$, we can follow the same procedure as above and obtain a result
like Theorem~\ref{eigen}.
In other words, given $j$, the corresponding FBVP can be rewritten in the form \eqref{eqhamm-0} for a suitable choice of $\widehat\psi$, $k$, $\gamma_3$
(recall that $\gamma_i$, $i=0,1,2$, are needed in \eqref{def-psihat-new}, i.e. the definition of $\widehat\psi$).
Hereafter, with  a slight abuse of notation, for simplicity we use the same notation for $k$ and $\gamma_i$ for all the FBVPs.

$\bullet$ \textbf{The case $j=0$ in \eqref{eq-fbc-all}.}\ Firstly we consider the case, in which the nonlocal BC \eqref{fbc-intro} involves $u(1)$.
\begin{equation} \label{eq.BC0}
 \begin{cases}
u^{(4)}(t)+\lambda F(t,u_t)=0,\ & t \in [0,1], \\
u(t)=\psi(t),\  & t \in [-r,0], \\
u(1)=\lambda B[u],
 \end{cases}
\end{equation}
We build up, as above, an integral equation equivalent to \eqref{eq.BC0}.

Given a continuous function $y$, the BVP:
\begin{equation*} \label{eq.BC0-omog}
 \begin{cases}
-u^{(4)}(t)=y(t),\  t \in [0,1], \\
u(0)=u'(0)=u''(0)=u(1)=0,
 \end{cases}
\end{equation*}
has the unique solution
$$
u(t)= \int_{0}^{1} k(t,s)y(s)ds,
$$
where the Green's function is
$$
k(t,s)= \frac{1}{6}\begin{cases}t^3(1-s)^3,\ &0\leq t\leq s\leq 1,\\ s(1-t)(s^2t^2-3st^2+3t^2+s^2t-3st+s^2),\ &0\leq s\leq t\leq 1.
\end{cases}
$$
In a similar way as above, 
direct computations show that the following problems have the corresponding solutions:
\begin{equation*} \label{eq.BC0-gamma0}
 \begin{cases}
u^{(4)}(t)=0,\  t \in [0,1], \\
u(0)=1, \ u'(0)=u''(0)=u(1)=0,
 \end{cases}
\end{equation*}
is solved by
$\gamma_0(t)=1-t^3$;
\begin{equation*} \label{eq.BC0-gamma1}
 \begin{cases}
u^{(4)}(t)=0,\  t \in [0,1], \\
u'(0)=1, \ u(0)=u''(0)=u(1)=0,
 \end{cases}
\end{equation*}
is solved by 
$\gamma_1(t)=t-t^3$;
\begin{equation*} \label{eq.BC3-gamma2}
 \begin{cases}
u^{(4)}(t)=0,\  t \in [0,1], \\
u''(0)=1, \ u(0)=u'(0)=u(1)=0.
 \end{cases}
\end{equation*}
 is solved by
$\gamma_2(t)=\frac12 t^2(1-t)$;
while for the condition at $t=1$ we have
\begin{equation*} \label{eq.BC3-gamma3}
 \begin{cases}
u^{(4)}(t)=0,\  t \in [0,1], \\
u(0)=u'(0)=u''(0)=0, \ u(1)=1,
 \end{cases}
\end{equation*}
which is solved by
$\gamma_3(t)= t^3$.
We can apply then apply the same procedure as in the previous case.

$\bullet$ \textbf{The case $j=1$ in \eqref{eq-fbc-all}.}\ A similar approach can be followed for the case
\begin{equation*} \label{eq.BC1}
 \begin{cases}
u^{(4)}(t)+\lambda F(t,u_t)=0,\ & t \in [0,1], \\
u(t)=\psi(t),\  & t \in [-r,0], \\
u'(1)=\lambda B[u].
 \end{cases}
\end{equation*}
In fact,  given a continuous function $y$, the BVP:
\begin{equation*} \label{eq.BC1-omog}
 \begin{cases}
-u^{(4)}(t)=y(t),\  t \in [0,1], \\
u(0)=u'(0)=u''(0)=u'(1)=0,
 \end{cases}
\end{equation*}
has the unique solution
$$
u(t)= \int_{0}^{1} k(t,s)y(s)ds,
$$
where the Green's function is
$$
k(t,s)= \frac{1}{6}\begin{cases}t^3(1-s)^2,\ &0\leq t\leq s\leq 1,\\ s(st^3-2t^3+3t^2-3st+s^2),\ &0\leq s\leq t\leq 1.
\end{cases}
$$

Elementary computations show that the following problems have the corresponding solutions:
\begin{equation*} \label{eq.BC1-gamma0}
 \begin{cases}
u^{(4)}(t)=0,\  t \in [0,1], \\
u(0)=1, \ u'(0)=u''(0)=u'(1)=0,
 \end{cases}
\end{equation*}
is solved by
$\gamma_0(t)=1$;
\begin{equation*} \label{eq.BC1-gamma1}
 \begin{cases}
u^{(4)}(t)=0,\  t \in [0,1], \\
u'(0)=1, \ u(0)=u''(0)=u'(1)=0,
 \end{cases}
\end{equation*}
is solved by 
$\gamma_1(t)=t-\frac13t^3$;
\begin{equation*} \label{eq.BC1-gamma2}
 \begin{cases}
u^{(4)}(t)=0,\  t \in [0,1], \\
u''(0)=1, \ u(0)=u'(0)=u'(1)=0,
 \end{cases}
\end{equation*}
 is solved by
$\gamma_2(t)= t^2(\frac12-\frac13t)$;
while for the condition at $t=1$ we have
\begin{equation*} \label{eq.BC1-gamma3}
 \begin{cases}
u^{(4)}(t)=0,\  t \in [0,1], \\
u(0)=u'(0)=u''(0)=0, \ u'(1)=1,
 \end{cases}
\end{equation*}
which is solved by
$\gamma_3(t)=\frac12 t^3$.

$\bullet$ \textbf{The case $j=2$ in \eqref{eq-fbc-all}.}\ For completeness we discuss also the case
\begin{equation*} \label{eq.BC2}
 \begin{cases}
u^{(4)}(t)+\lambda F(t,u_t)=0,\ & t \in [0,1], \\
u(t)=\psi(t),\  & t \in [-r,0], \\
u''(1)=\lambda B[u].
 \end{cases}
\end{equation*}
Notice that, given a continuous function $y$, the BVP:
\begin{equation*} \label{eq.BC2-omog}
 \begin{cases}
-u^{(4)}(t)=y(t),\  t \in [0,1], \\
u(0)=u'(0)=u''(0)=u''(1)=0,
 \end{cases}
\end{equation*}
has the unique solution
$$
u(t)= \int_{0}^{1} k(t,s)y(s)ds,
$$
where the Green's function is
$$
k(t,s)= \frac{1}{6}\begin{cases}t^3(1-s)^2,\ &0\leq t\leq s\leq 1,\\ s(-t^3+3t^2-3st+s^2),\ &0\leq s\leq t\leq 1.
\end{cases}
$$

Elementary computations in this case show that the following problems have the corresponding solutions:
\begin{equation*} \label{eq.BC2-gamma0}
 \begin{cases}
u^{(4)}(t)=0,\  t \in [0,1], \\
u(0)=1, \ u'(0)=u''(0)=u''(1)=0,
 \end{cases}
\end{equation*}
is solved by
$\gamma_0(t)=1$;
\begin{equation*} \label{eq.BC2-gamma1}
 \begin{cases}
u^{(4)}(t)=0,\  t \in [0,1], \\
u'(0)=1, \ u(0)=u''(0)=u''(1)=0,
 \end{cases}
\end{equation*}
is solved by 
$\gamma_1(t)=t$;
\begin{equation*} \label{eq.BC2-gamma2}
 \begin{cases}
u^{(4)}(t)=0,\  t \in [0,1], \\
u''(0)=1, \ u(0)=u'(0)=u''(1)=0,
 \end{cases}
\end{equation*}
 is solved by
$\gamma_2(t)= \frac12 t^2(1-\frac13t)$;
while for the condition at $t=1$ we have
\begin{equation*} \label{eq.BC2-gamma3}
 \begin{cases}
u^{(4)}(t)=0,\  t \in [0,1], \\
u(0)=u'(0)=u''(0)=0, \ u''(1)=1,
 \end{cases}
\end{equation*}
which is solved by
$\gamma_3(t)=\frac16 t^3$.

Notice that in all the cases illustrated above the kernels $k$ and the functions $\gamma_i$ are nonnegative.

\section{Examples}

We remark that
our theory can be applied to \emph{delay differential equations}~(DDEs). Namely, 
 let $f:[0,1]\times\R_+\times\R\times\R \times\R_+\times\R\times\R \to [0,\infty)$ be a given Carath\'eodory map.
Consider the parametrized fourth-order DDE with three time-lags
 \begin{equation}\label{eqdelay4}
u^{(4)}(t)+\lambda f(t,u(t),u'(t),u''(t),u(t-r_0), u'(t-r_1),u''(t-r_2)),\ t \in [0,1],
\end{equation}
where the fixed (possibly different) delays $r_i$ are positive, for $i=0,\dots,2$.
We can apply the techniques developed in this paper to the equation 
\eqref{eqdelay4} with initial condition \eqref{eq-ic} along with one of the BCs \eqref{eq-fbc-all}.

To see this, observe that \eqref{eqdelay4} is a special case of the functional equation \eqref{eq-4ord}, in which taking $r:= \max\{r_i,i=0,\dots2\}$,
 the operator $F: [0,1] \times C^2([-r, 0], \R) \to [0,\infty)$ is defined by
$$
F(t,\phi)=f(t,\phi(0),\phi'(0),\phi''(0),\phi(-r_0),\phi'(-r_1),\phi''(-r_2)).
$$ 
Such an operator satisfies the above Carath\'eodory-type conditions if the map $f$ satisfies analogous properties; namely,

$\bullet$\ for each $R>0$, there exists $\varphi^*_{R} \in
L^{\infty}[0,1]$ such that{}
\begin{align*}
 f(t,u,v,w, \xi, \eta, \zeta ) \le \varphi^*_{R}(t) \; & \text{ for all  } \; (u,v,w, \xi, \eta, \zeta) \in \R_+\times\R\times\R \times\R_+\times\R\times\R \\
\;\text{  with } \; 0\le u,\xi \le R,\; & |v|,|w|,|\eta|,|\zeta| \le R,\;\text{ and a.\,e. } \; t\in [0,1].
\end{align*}

In the following illustrating example (cfr.\ \cite{acgi2, giaml22}) we consider a specific DDE of type
\eqref{eqdelay4}, with $r= \max\{r_i,i=0,\dots2\}=1/2$. Observe in particular that the vertex $\widehat\psi$ of the affine cone, to which the solutions belong, is obtained by a $C^2$-prolongation of the datum $\psi$ after $t_0=0$.

\begin{ex}\label{examp}
We  consider the family of FBVPs
\begin{equation}\label{eqdiffex}
u^{(4)}+\lambda te^{u(t)+\bigl(u''\bigl(t-\frac13\bigl)\bigr)^2}\Bigl(1+(u'(t))^2+\Bigl(u\Bigl(t-\frac12\Bigr)\Bigr)^2+\Bigl(u''\Bigl(t-\frac14\Bigr)\Bigr)^2\Bigr),\ t \in (0,1),
\end{equation}
with the initial condition
\begin{equation}\label{inex}
u(t)=\psi(t), t \in \Bigl[-\frac12,0\Bigr],
\end{equation}
with $\psi(t)=H(-t) \cdot (1-\cos t)$,
and one of the four BCs~\eqref{eq-fbc-all}.
For example we choose $j=3$, so that the functional BC is 
\begin{equation}\label{fbc-ex}
u'''(1)=\lambda B[u],
\end{equation}
% in ~\eqref{eq-fbc-all}, and
where we fix
\begin{equation}\label{dta-dep}
B[u]=\frac{1}{1+(u(\frac{1}{2}))^2}+\int_{-\frac12}^1t^3(u''(t))^2\, dt.
\end{equation}
Thus the function $\widehat \psi$ is given by
$$\widehat\psi(t) =
\begin{cases}
1-\cos t,\ &  -\frac12 \leq t \leq 0, \\
\frac12 t^2,\ & 0< t \leq 1.
\end{cases}
$$

Now choose $\rho\in (0,+\infty)$. We may take 
$$\underline{\eta}_{\rho}(t)=\frac{1}{1+\rho^2}, \underline{\delta}_{\rho}(t)=t.$$ 

Therefore we have
$$
\sup_{t\in [0,1]}\Bigl\{ \frac{\frac12 t^2 }{1+\rho^2}+\int_{0}^{1}  k(t,s)t\,ds\Bigr\}\geq \frac{1}{6(1+\rho^2)} > 0, 
$$
which implies that \eqref{condc} is satisfied for every $\rho \in (0,+\infty)$.

Thus we can apply Theorem~\ref{eigen} obtaining uncountably many pairs of solutions and parameters $(u_{\rho}, \lambda_{\rho})$ for the FBVP
 \eqref{eqdiffex}--\eqref{inex}--\eqref{fbc-ex}.
\end{ex}

\section*{Acknowledgements}
The authors were partially supported by
 the Gruppo Nazionale per l'Analisi Matematica, la Probabilit\`a e le loro Applicazioni (GNAMPA) of the Istituto Nazionale di Alta Matematica (INdAM).
G.~Infante is a member of the UMI Group TAA  ``Approximation Theory and Applications''.

\end{document}